\documentclass[12pt]{amsart}

\usepackage{amsmath}
\usepackage{amssymb}
\usepackage{ascmac}
\usepackage{amsthm}

\makeatletter

\@addtoreset{equation}{section}
\makeatother 

\theoremstyle{plain}
\newtheorem{thm}{Theorem}[section]
\newtheorem*{thm*}{Theorem}
\newtheorem{prop}[thm]{Proposition}
\newtheorem{lem}[thm]{Lemma}

\theoremstyle{remark}

\newcommand{\vol}{\operatorname{vol}}

\newcommand{\diam}{\operatorname{Diam}}

\usepackage{latexsym}

\title[Distribution of Neumann eigenvalues]{A note on the distribution of Neumann eigenvalues of the Laplacian on a Euclidean convex domain}

\author{Kei Funano}
\address{Division of Mathematics \& Research Center for Pure and Applied Mathematics, Graduate School of Information Sciences, Tohoku University, 6-3-09 Aramaki-Aza-Aoba, Aoba-ku, Sendai 980-8579, Japan}
\email{kfunano@tohoku.ac.jp}

\date{\today}

\begin{document}
\maketitle

\begin{abstract}
We establish two universal inequalities for Neumann eigenvalues of the Laplacian on a Euclidean convex domain.
\end{abstract}

%%%%%%%%%%%%%%%%%%%%%%%%%%%%
%%%%%%%%%%%%%%%%%%%%%%%%%%%%
%%%%%%%%%%%%%%%%%%%%%%%%%%%%

%%%%%%%%%%%%%%%%%%%%%%%%%%%
%%%%%%%%%%%%%%%%%%%%%%%%%%%
%%%%%%%%%%%%%%%%%%%%%%%%%%%

%%%%%%%%%%%%%%%%%%%%%%%%

%%%%%%%%%%%%%%%%%%%%%%%%%%%%
%%%%%%%%%%%%%%%%%%%%%%%%%%%%
%%%%%%%%%%%%%%%%%%%%%%%%%%%%

%%%%%%%%%%%%%%%%%%%%%%%%

%%%%%%%%%%%%%%%%%%%%%%%%%%%%
%%%%%%%%%%%%%%%%%%%%%%%%%%%%
%%%%%%%%%%%%%%%%%%%%%%%%%%%%
\section{Introduction}

Our primary concern in this paper is universal inequalities for eigenvalues of the Laplacian on a bounded domain in Euclidean space. For Dirichlet eigenvalues, the study of universal inequlities was pioneered by Payne, P\'{o}lya, and Weinberger \cite{PPW}. Since their seminal work, numerous universal inequalities have been established for the Dirichlet case (Refer to \cite{A}). A distinguishing feature of these Dirichlet universal inequalities is that they hold for arbitrary domains without requiring any specific geometric conditions.

In contrast, the situation for Neumann eigenvalues is vastly different. A celebrated result by Colin de Verdi\`{e}re \cite{CdV} demonstrates that for any given finite increasing sequence of non-negative numbers, there exists a domain whose initial Neumann eigenvalues exactly coincide with this sequence. Consequently, without imposing geometric restrictions on the domain, one cannot expect any general universal inequalities for Neumann eigenvalues.

For a bounded domain $\Omega$ in $\mathbb{R}^n$ with a sufficently regular boundary (e.g., Lipschitz boundary) we denote its Neumann eigenvalues of the Laplacian in non-decreasing order (repeated by multiplicity) as 
\begin{align*}
     0=\mu_0(\Omega)<\mu_1(\Omega)\le\mu_2(\Omega)\le \dots  \le \mu_k(\Omega)\le \mu_{k+1}(\Omega)\le \dots \to\infty.
\end{align*}

Our main theorem in this paper is the following.
\begin{thm}\label{MTHM}There exists a universal numeric constant $c>0$ which satisfies the following. Let $\Omega$ be a bounded convex domain in $\mathbb{R}^n$. Then for any $k\geq l$ we have
    \begin{align*}
       \frac{1}{cn^{10}} \Big(\frac{k}{l}\Big)^{\frac{2}{n}}\mu_l(\Omega) \leq \mu_k(\Omega)\leq c n^{10} \Big(\frac{k}{l}\Big)^{2}\mu_l(\Omega).
    \end{align*}
\end{thm}

Let us mention the related known inequalities between the eigenvalues. Below $c_1,c_2,\cdots$ will be concrete universal constants. In \cite[Theorem 1.1 and Remark 1.1]{F3} building on theorems by Milman (\cite{Mi}) and Cheng-Li (\cite{CL}) the author showed the universal inequality $\mu_k(\Omega)\geq  c_1 k^{\frac{2}{n}}\mu_1(\Omega)$ for any bounded convex domain $\Omega$ in $\mathbb{R}^n$ and this is sharp up to a multiplicative universal constant factor. The first inequality in the above theorem generalizes this inequality but the dimensional constant $n^{10}$ appears in the inequality. 
In \cite[Theorem 1.1]{Liu}, Liu showed the sharp dimension-free universal inequality $\mu_k(\Omega)\leq c_2k^2\mu_1(\Omega)$ for any bounded convex domain $\Omega$ in $\mathbb{R}^n$.
In \cite[Theorem 1.2]{F2} the author proved that 
\begin{align}\label{univ}
    \mu_{k+1}(\Omega)\leq c_3 n^4\mu_k(\Omega)
\end{align}under the same situation of that of Liu's. The second inequality in the above theorem unifies the Liu inequality and the inequality (\ref{univ}) but it includes the dimensional constant $n^{10}$ which is even bigger than $n^4$ appeared in (\ref{univ}).

In \cite[Theorems 1.1 and 1.2]{F4} the author proved Theorem \ref{MTHM} for planar convex domains. The proof of Theorem \ref{MTHM} largely follows that given in the paper.
We first reduce the problem to the case of an orthotope utilizing domain monotonicity and standard convex geometry. 
Secondly we find a 'nice' convex partition of the orthotope to give a lower bound for a Neumann eigenvalue of the orthotope in terms of the first eigenvalue of the parts. In \cite{F4} we only constructed such a partition in the planar case and in this paper we inductively construct a partition to generalize for higher dimension, see Lemmas \ref{Mlem1} and \ref{Mlem2}.

\section{Preliminaries}

In this section we collect several results which will be used in the proof of the main theorem.

We use the following key proposition where Hatcher showed that bounded convex domains can be approximated by an orthotope in some sense. It is the consequence of the John theorem which asserts that bounded convex domains can be approximated by ellipsoids (\cite[Theorem III]{J}). 

\begin{prop}[{\cite[Proposition 2.3]{H}}]\label{Hacher lem} Let $\Omega$ be a bounded convex domain in $\mathbb{R}^n$. Choosing the coordinate axes appropriately if necessary we can find a sequence $a_1,a_2,\cdots, a_n$ of positive real numbers such that  
    \begin{align*}
       \Big[-\frac{a_1}{\sqrt{n}},\frac{a_1}{\sqrt{n}}\Big]\times \cdots \times \Big[-\frac{a_n}{\sqrt{n}},\frac{a_n}{\sqrt{n}}\Big]\subseteq  \Omega\subseteq [-a_1 n, a_1 n]\times \cdots [-a_n n, a_n n].
    \end{align*}
\end{prop}

The above proposition together with the next theorem suggest us that the proof of the main theorem reduces to the case where the domain is an orthotope.
The next theorem states that the smaller domain has a larger eigenvalue up to a multiplicative constant factor under convexity assumption.
\begin{thm}[{\cite[Theorem 1.1]{F}}]\label{domain monotonicity}There exists a universal constant $c>0$ which satisfies the following. For any two
bounded convex domains $\Omega\subseteq \Omega'$ in $\mathbb{R}^n$ their Neumann eigenvalues of the Laplacian satisfy
    \begin{align*}
        \mu_k(\Omega')\leq c n^2\mu_k(\Omega).
    \end{align*}The constant $c$ can be choosen as $c=(92)^2$.
\end{thm}

Under the assumption of convexity, Neumann eigenvalues of the Laplacian are related with the diameter as follows:
\begin{thm}[{Payne–Weinberger, \cite[(1.9)]{PW}}]\label{Diamthm}Suppose that $\Omega$ is a bounded convex domain in $\mathbb{R}^n$. Then we have $\mu_1(\Omega)\geq \frac{\pi^2}{(\diam \Omega)^2}$.
\end{thm}

We also use the following upper bound for Neumann eigenvalues in terms of volume.   
\begin{thm}[{Kr\"{o}ger, \cite[Corollary 2]{Kr2}}]\label{Kr ineq2}Let $\Omega$ be a (not necessarily convex) bounded domain in $\mathbb{R}^n$ with Lipschitz boundary. Then for any natural number $k$ we have
\begin{align*}
    \mu_k(\Omega)\leq (2\pi)^2\Big(\frac{n+2}{2}\Big)^{\frac{2}{n}}\Big(\frac{k}{\omega_n \vol \Omega}\Big)^{\frac{2}{n}}\lesssim \Big(\frac{k}{\omega_n \vol \Omega}\Big)^{\frac{2}{n}},
\end{align*}where $\omega_n$ is the volume of a unit ball in $\mathbb{R}^n$.
\end{thm}

Let $\Omega$ be a bounded domain in a Euclidean space and $\{\Omega_i\}_{i=1}^{k}$ be a finite partition of $\Omega$ by subdomains;
    $\Omega=\bigcup_i \Omega_i$ and $\vol (\Omega_i\cap \Omega_j)=0$ for different $i\neq j$. Buser gave a lower bound of eigenvalues of the Laplacian in terms of the first eigenvalue of the parts of a finite partition:

\begin{thm}[{Buser, \cite[Theorem 8.2.1]{B}}]\label{Buserthm}Under the above situation we have
    \begin{align*}
        \mu_k(\Omega)\geq \min_i \mu_1(\Omega_i).
    \end{align*}
\end{thm}

\section{Proof of the Main Theorem}
In this section we prove the main theorem. 
The following two lemmas are keys to prove the theorem. 
\begin{lem}\label{Mlem1}There exists a universal numeric constant $c>0$ satisfying the following. For any orthotope $\Omega$ in $\mathbb{R}^n$ and for any $k\geq l$ there exists a convex partition $\{\Omega_i\}_{i=1}^{l'}$ of $\Omega$ such that $l'\leq l$ and 
    \begin{align*}
        \diam \Omega_i \leq cn^{\frac{3}{2}}\frac{k}{l\sqrt{\mu_k(\Omega)}}.
    \end{align*}
\end{lem}
\begin{lem}\label{Mlem2}There exists a universal numeric constant $c>0$ satisfying the following. For any orthotope $\Omega$ in $\mathbb{R}^n$ and for any $k\geq l$ there exists a convex partition $\{\Omega_i\}_{i=1}^{k'}$ of $\Omega$ such that $k'\leq k$ and 
    \begin{align*}
        \diam \Omega_i \leq cn^{\frac{3}{2}}\Big(\frac{l}{k}\Big)^{\frac{1}{n}}\frac{1}{\sqrt{\mu_l(\Omega)}}.
    \end{align*}
\end{lem}

\begin{proof}[Proof of Theorem \ref{MTHM}]By Proposition \ref{Hacher lem} there exists an orthotope $R$ such that after choosing the coordinate axes we have
\begin{align*}
    \frac{1}{\sqrt{n}}R\subseteq \Omega \subseteq n R.
\end{align*}
By Lemma \ref{Mlem1} there exists a convex partition $\{\Omega_i\}_{i=1}^{l'}$ of $R$ such that $l'\leq l$ and 
\begin{align*}
    \diam \Omega_i \leq cn^{\frac{3}{2}}\frac{k}{l\sqrt{\mu_k(R)}}\ \text{ for all }i.
\end{align*}
Thus Buser's theorem (Theorem \ref{Buserthm}) and the Payne-Weinberger inequality (Theorem \ref{Diamthm}) imply that
\begin{align*}
    \mu_l(R)\geq \mu_{l'}(R)\geq \min_i \mu_1(\Omega_i)\geq \min_i \frac{\pi^2}{(\diam \Omega_i)^2} \geq  \frac{l^2\mu_k(R)}{c^2n^3k^2}.
\end{align*}By virtue of Theorem \ref{domain monotonicity} we obtain
\begin{align*}
    \mu_l(\Omega)\geq \frac{\mu_l(nR)}{(92 n)^{2}}=\frac{\mu_l(R)}{(92)^{2}n^{4}}\geq \frac{l^2\mu_k(R)}{(92)^2c^2n^7k^2}= \ &\frac{l^2\mu_k(\frac{1}{\sqrt{n}}R)}{(92)^2c^2n^8k^2}\\ \geq \ &\frac{l^2\mu_k(\Omega)}{(92)^4c^2n^{10}k^2},
\end{align*}which implies the second inequality of Theorem \ref{MTHM}. In the same way the first inequality of the theorem follows utlizing Lemma \ref{Mlem2}. This completes the proof of the theorem.

\end{proof}
\begin{proof}[Proof of Lemma \ref{Mlem1}]
    Take a universal constant $C>0$ such that 
    \begin{align*}
         C\geq \Big((2\pi)^n\cdot \frac{n+2}{2} \cdot \frac{1}{(\omega_n)^2}\Big)^{\frac{1}{n}}\cdot \frac{1}{n}\sim \text{ const. }
    \end{align*}for any $n\geq 2$. Also take a universal constant $c>0$ such that
    \begin{align*}
        c\geq 4C \text{ and }c>\frac{2\sqrt{n}}{\sqrt{n}-1} \text{ for any }n\geq 2.
    \end{align*}

    For such $c$ we prove the lemma by induction on $n$. The case where $n=1$ is trivial. 

    Suppose that the claim of the lemma holds for any orthotope $\Omega$ in $\mathbb{R}^{n-1}$, that is, suppose that for such $\Omega$ we can always find a convex partition $\{\Omega_i\}_{i=1}^{l'}$ of $\Omega$ such that $l'\leq l$ and $\diam \Omega_i\leq c(n-1)^{\frac{3}{2}}\frac{k}{l\sqrt{\mu_k(\Omega)}}$.

    Let $\Omega$ be an orthotope in $\mathbb{R}^n$. After choosing the coordinate axes appropriately we may assume that $\Omega=[-a_1,a_1]\times [-a_2,a_2]\times \cdots \times [-a_n,a_n]$ and $a_1\leq a_2\leq \cdots \leq a_n$.

    Put $r:=\frac{Cnk}{l\sqrt{\mu_k(\Omega)}}$. We first consider the case where $a_1\leq 2r$. Set $\Omega':=[-a_2,a_2]\times [-a_3,a_3]\times \cdots \times [-a_n,a_n]$. By the assumption of the induction there exists a convex partition $\{\Omega_i'\}_{i=1}^{l'}$ of $\Omega'$ such that $l'\leq l$ and $\diam \Omega_i'\leq c(n-1)^{\frac{3}{2}}\frac{k}{l\sqrt{\mu_k(\Omega')}}$. Note that $\mu_k(\Omega)\leq \mu_k(\Omega')$ (This follows from the concrete expression of eigenvalues of an orthotope). Thus we get $\diam \Omega_i'\leq c(n-1)^{\frac{3}{2}}\frac{k}{l\sqrt{\mu_k(\Omega)}}$. Setting $\Omega_i:=[-a_1,a_1]\times \Omega_i'$ we see that $\{\Omega_i\}_{i=1}^{l'}$ covers $\Omega$. We also have 
    \begin{align*}
        \diam \Omega_i\leq \sqrt{4^2C^2n^2+c^2(n-1)^3}\frac{k}{l\sqrt{\mu_k(\Omega)}}\leq cn^{\frac{3}{2}}\frac{k}{l\sqrt{\mu_k(\Omega)}},
    \end{align*}where the second inequality follows from our choice of the constants $c$ and $C$. Therefore $\{\Omega_i\}_{i=1}^{l'}$ is a desired convex partition of $\Omega$.

    Let us consider the case where $a_1>2r$. In this case we put 
    \begin{align*}
        \widetilde{\Omega}:=\{x\in \Omega \mid d(x,\partial \Omega)\geq r\},
    \end{align*}where 
    \begin{align*}
        d(x,\partial \widetilde{\Omega}):=\inf \{|x-y| \mid y\in \partial \widetilde{\Omega}\}.
    \end{align*}
    Note that $\widetilde{\Omega}$ is a nonempty (noncollapsed) orthotope. Let $\{x_i\}_{i=1}^{l'}$ be a maximal $(2r)$-separated points in $\widetilde{\Omega}$. Since $\{B(x_i,r)\}_{i=1}^{l'}$ are disjoint balls in $\Omega$, we have
    \begin{align*}
       \frac{l'\omega_n(Cnk)^n}{l^n\mu_k(\Omega)^{\frac{n}{2}}}= l'\omega_n r^n =\sum_{i=1}^{l'} \vol B(x_i,r)\leq \vol \Omega.
    \end{align*}
    Applying Theorem \ref{Kr ineq2} to the above inequality we have
    \begin{align*}
        l'\leq (2\pi)^n\frac{n+2}{2}\cdot \frac{1}{(\omega_n)^2}\cdot \frac{l^{n-1}}{C^n n^nk^{n-1}}\cdot l\leq \frac{l^{n-1}}{k^{n-1}}\cdot l \leq l,
    \end{align*}where the second inequality follows from our choice of $C$ and $k\geq l$ implies the last inequality.

    The maximality of the points $\{x_i\}_{i=1}^{l'}$ gives $\widetilde{\Omega}\subseteq \bigcup_{i=1}^{l'}B(x_i,2r)$. Also observe that the $\sqrt{n}r$-neighbourhood of $\widetilde{\Omega}$ covers $\Omega$ (note that the $r$-neighbourhood of $\widetilde{\Omega}$ does not cover $\Omega$). We thereby have $\Omega\subseteq \bigcup_{i=1}^{l'}B(x_i,(2+\sqrt{n})r)$. Therefore if we introduce the \emph{Voronoi cells}
    \begin{align*}
        \Omega_i:=\{x\in \Omega \mid |x-x_i|\leq |x-x_j| \text{ for all }j\neq i\},
    \end{align*}then $\{\Omega_i\}_{i=1}^{l'}$ is a convex partition of $\Omega$. Since $\Omega_i\subseteq B(x_i,(2+\sqrt{n})r)$ our choice of $c$ implies
    \begin{align*}
        \diam \Omega_i\leq 2(2+\sqrt{n})r\leq c n^{\frac{3}{2}}\frac{k}{l\sqrt{\mu_k(\Omega)}}.
    \end{align*}Therefore we obtain the desired convex partition $\{\Omega_i\}_{i=1}^{l'}$ of $\Omega$. This completes the proof.
\end{proof}
Although the same method in the above proof applies to the proof of Lemma \ref{Mlem2}, we give the proof for the completeness of this paper. 
\begin{proof}[Proof of Lemma \ref{Mlem2}]
We take constants $c,C>0$ as in the proof of Lemma \ref{Mlem1}, that is, $C>0$ is a constant such that 
    \begin{align*}
         C\geq \Big((2\pi)^n\cdot \frac{n+2}{2} \cdot \frac{1}{(\omega_n)^2}\Big)^{\frac{1}{n}}\cdot \frac{1}{n}
    \end{align*}for any $n\geq 2$ and $c>0$ is a constant such that
    \begin{align*}
        c\geq 4C \text{ and }c>\frac{2\sqrt{n}}{\sqrt{n}-1} \text{ for any }n\geq 2.
    \end{align*}
    
For the constant $c>0$ we prove the lemma by induction on $n$. The case where $n=1$ is trivial.

Suppose that the claim of the lemma holds for any orthotope $\Omega$ in $\mathbb{R}^{n-1}$, that is, suppose that for such $\Omega$ we can always find a convex partition $\{\Omega_i\}_{i=1}^{k'}$ of $\Omega$ such that $k'\leq k$ and $\diam \Omega_i\leq c(n-1)^{\frac{3}{2}}\big(\frac{l}{k}\big)^{\frac{1}{n-1}}\frac{1}{\sqrt{\mu_l(\Omega)}}$.

Let $\Omega$ be an orthotope in $\mathbb{R}^n$. We may assume that $\Omega=[-a_1,a_1]\times [-a_2,a_2]\times \cdots \times [-a_n,a_n]$ and $a_1\leq a_2\leq \cdots \leq a_n$. 

Put $r:=Cn \big(\frac{l}{k}\big)^{\frac{1}{n}}\frac{1}{\sqrt{\mu_l(\Omega)}}$. The proof of the existence of a desired convex partition of $\Omega$ is divided into two cases. The first case is the case where $a_1\leq 2r$. Set $\Omega':=[-a_2,a_2]\times [-a_3,a_3]\times \cdots \times [-a_n,a_n]$. By the assumption of the induction there exists a convex partition $\{\Omega_i'\}_{i=1}^{k'}$ of $\Omega'$ such that $k'\leq k$ and $\diam \Omega_i'\leq c(n-1)^{\frac{3}{2}}\big(\frac{l}{k}\big)^{\frac{1}{n-1}}\frac{1}{\sqrt{\mu_l(\Omega')}}$. Since $\mu_l(\Omega')\geq \mu_l(\Omega)$ and $k\geq l$ we get $\diam \Omega_i'\leq c(n-1)^{\frac{3}{2}}\big(\frac{l}{k}\big)^{\frac{1}{n}}\frac{1}{\sqrt{\mu_l(\Omega)}}$. Setting $\Omega_i:=[-a_1,a_1]\times \Omega_i'$ we have 
    \begin{align*}
        \diam \Omega_i\leq \sqrt{4^2C^2n^2+c^2(n-1)^3}\Big(\frac{l}{k}\Big)^{\frac{1}{n}}\frac{1}{\sqrt{\mu_l(\Omega)}}\leq cn^{\frac{3}{2}}\Big(\frac{l}{k}\Big)^{\frac{1}{n}}\frac{1}{\sqrt{\mu_k(\Omega)}}.
    \end{align*}Hence $\{\Omega_i\}_{i=1}^{k'}$ is a desired convex partition of $\Omega$.

The second case is the case where $a_1>2r$. In this case as in the proof of Lemma \ref{Mlem1} we consider the subdomain
\begin{align*}
    \widetilde{\Omega}:=\{x\in \Omega \mid d(x,\partial\Omega)\geq r\}
\end{align*}and take a maximal $2r$-separated points $\{x_i\}_{i=1}^{k'}$ in $\widetilde{\Omega}$. If we show $k'\leq k$ then as in the proof of Lemma \ref{Mlem1} we get the desired convex partition $\{\Omega_i\}_{i=1}^{k'}$. Since $\{B(x_i,r)\}_{i=1}^{k'}$ are disjoint balls in $\Omega$, we have
\begin{align*}
k'\omega_n r^n=\sum_{i=1}^{k'}\vol B(x_i,r)\leq \vol \Omega.   
\end{align*}Applying Kr\"{o}ger's inequality (Theorem \ref{Kr ineq2}) to the above inequality we obtain
\begin{align*}
    k'\leq \frac{\vol \Omega}{w_n R^n}\leq\frac{(2\pi)^n k}{C^n n^n \omega_n^2 }\cdot \frac{n+2}{2}\leq k,
\end{align*}where the last inequality follows from our choice of $C$. This completes the proof of the lemma. 
\end{proof}

\emph{Acknowledgement.}This work was supported by JSPS KAKENHI Grant Number JP24K06731.

%%%%%%%%%%%%%%%%%%%%%%%%%%%%

\end{document}